\documentclass[12pt, reqno]{amsart}
\usepackage{amsmath, amsthm, amscd, amsfonts, amssymb, graphicx, color}
\usepackage[bookmarksnumbered, colorlinks, plainpages]{hyperref}
\usepackage{slashbox}
\newtheorem{theorem}{Theorem}[section]
\newtheorem{lemma}[theorem]{Lemma}

\newtheorem{corollary}[theorem]{Corollary}
\theoremstyle{definition}

\newtheorem{example}[theorem]{Example}

\theoremstyle{remark}
\newtheorem{remark}[theorem]{Remark}
\numberwithin{equation}{section}

\begin{document}
\setcounter{page}{1}

\title[On weighted means and their inequalities]{On weighted means and their inequalities}

\author[M. Ra\"{\i}ssouli and S. Furuichi]{Mustapha Ra\"{\i}ssouli$^{1,2}$ and Shigeru Furuichi$^{3}$}

\address{$^{1}$ Department of Mathematics, Science Faculty, Taibah University,
Al Madinah Al Munawwarah, P.O.Box 30097, Zip Code 41477, Saudi Arabia.}

\address{$^{2}$ Department of Mathematics, Science Faculty, Moulay Ismail University, Meknes, Morocco.}

\address{$^{3}$ Department of Information Science, College of Humanities and Sciences, Nihon University,
3-25-40, Sakurajyousui, Setagaya-ku, Tokyo, 156-8550, Japan}

\email{\textcolor[rgb]{0.00,0.00,0.84}{raissouli.mustapha@gmail.com}}
\email{\textcolor[rgb]{0.00,0.00,0.84} {furuichi@chs.nihon-u.ac.jp}}
\
\subjclass[2010]{}


\date{Received: xxxxxx; Revised: yyyyyy; Accepted: zzzzzz.}

\begin{abstract}
In \cite{PSMA}, Pal et al. introduced some weighted means and gave some related inequalities by using an approach for operator monotone functions. This paper discusses the construction of these weighted means in a simple and nice setting that immediately leads to the inequalities established there. The related operator version is here immediately deduced as well. According to our constructions of the means, we study all cases of the weighted means from three weighted arithmetic/geometric/harmonic means, by the use of the concept such as stable and stabilizable means. Finally, the power symmetric means are studied and new weighted power means are given.
\end{abstract}

\maketitle

\section{\bf Introduction}

The mean inequalities arise in various contexts and attract many mathematicians by their developments and applications. It has been proved throughout the literature that the mean-theory is useful in theoretical point of view as well as in practical purposes.

\subsection{\bf Standard weighted means.} As usual, we understand by (binary) mean a map $m$ between two positive numbers such that $\min(a,b)\leq m(a,b)\leq\max(a,b)$ for any $a,b>0$. Continuous (symmetric/homogeneous) means are defined in the habitual way. If $m$ is a mean we define its dual by $m^*(a,b)=\left(m(a^{-1},b^{-1})\right)^{-1}$. It is easy to see that if $m$ is continuous,(resp. symmetric, homogeneous) then so is $m^*$. Of course, $m^{**}=m$ for any mean $m$. The means $(a,b)\longmapsto\min(a,b)$ and $(a,b)\longmapsto\max(a,b)$ are called trivial means. A mean $m$ is called strict if $m(a,b)=a$ implies $a=b$. The trivial means are not strict.

Among the standard means, we recall the arithmetic mean $a\nabla b=\dfrac{a+b}{2}$, the geometric mean $a\sharp b=\sqrt{ab}$, the harmonic mean $a!b=\dfrac{2ab}{a+b}$, the logarithmic mean $L(a,b)=\dfrac{b-a}{\log b-\log a}$ with $L(a,a)=a$ and the identric mean $I(a,b)=e^{-1}\big(b^b/a^a\big)^{1/(b-a)}$ with $I(a,a)=a$. It is easy to see that $H^*=A$ and $G^*=G$. The explicit expressions of $L^*$ and $I^*$ can be easily deduced from those of $L$ and $I$, respectively. All these means are strict. The following chain of inequalities is well-known in the literature
\begin{equation}\label{0}
a!b\leq I^*(a,b)\leq L^*(a,b)\leq a\sharp b\leq L(a,b)\leq I(a,b)\leq a\nabla b.
\end{equation}

Let $m_v$ be a binary map indexed by $v\in[0,1]$. We say that $m_v$ is a weighted mean if the following assertions are satisfied:\\
(i) $m_v$ is a mean, for any $v\in[0,1]$,\\
(ii) $m_{1/2}:=m$ is a symmetric mean,\\
(iii) $m_v(a,b)=m_{1-v}(b,a)$ for any $a,b>0$ and $v\in[0,1]$.

It is obvious that, (iii) implies (ii). The mean $m:=m_{1/2}$ is called the associated symmetric mean of $m_v$. It is not hard to check that, if $m_v$ is a weighted mean then so is $m_v^*$.

The standard weighted means are recalled in the following. The weighted arithmetic mean $a\nabla_vb=(1-v)a+vb$, the weighted geometric mean $a\sharp_vb=a^{1-v}b^v$ and the weighted harmonic mean $a!_vb=\big((1-v)a^{-1}+vb^{-1}\big)^{-1}$. For $v=1/2$ they coincide with $a\nabla b$, $a\sharp b$ and $a!b$, respectively. These weighted means satisfy
\begin{equation}\label{02}
a!_vb\leq a\sharp_vb\leq a\nabla_vb
\end{equation}
for any $a,b>0$ and $v\in[0,1]$. These weighted means are all strict provided $v\in(0,1)$.

\subsection{\bf Two weighted means.} Recently, Pal et al. \cite{PSMA} introduced a class of operator monotone functions from which they deduced other weighted means, namely the weighted logarithmic mean defined by
\begin{equation}\label{1}
L_v(a,b)=\dfrac{1}{\log a-\log b}\left(\dfrac{1-v}{v}\big(a-a^{1-v}b^v\big)+\dfrac{v}{1-v}\big(a^{1-v}b^v-b\big)\right)
\end{equation}
and the weighted identric mean given by
\begin{equation}\label{2}
I_v(a,b)=\frac{1}{e}\Big(a\nabla_vb\Big)^{\frac{(1-2v)(a\nabla_vb)}
{v(1-v)(b-a)}}\left(\dfrac{b^{\frac{vb}{1-v}}}{a^{\frac{(1-v)a}{v}}}\right)^{\frac{1}{b-a}}.
\end{equation}
One has $L_0(a,b):=\lim\limits_{v\downarrow0}L_v(a,b)=a$ and $L_1(a,b):=\lim\limits_{v\uparrow1}L_v(a,b)=b$, with similar equalities for $I(a,b)$. One can see that $L_v$ and $I_v$ satisfy the conditions (i),(ii) and (iii). For $v=1/2$, they coincide with $L(a,b)$ and $I(a,b)$, respectively.

It has been shown in \cite[Theorem 2.4, Theorem 3.1]{PSMA} that the inequalities
\begin{equation}\label{3}
a\sharp_vb\leq L_v(a,b)\leq(a\sharp_vb)\nabla(a\nabla_vb)\leq a\nabla_v b
\end{equation}
and
\begin{equation}\label{4}
a\sharp_vb\leq I_v(a,b)\leq a\nabla_v b
\end{equation}
hold for any $a,b>0$ and $v\in[0,1]$.

Otherwise, Furuichi and Minculete \cite{FM2020} gave a systematic study from which they obtained a lot of mean-inequalities involving $L_v(a,b)$ and $I_v(a,b)$. Some of their inequalities are refinements and reverses of \eqref{3} and \eqref{4}.\\

The outline of this paper will be organized as follows: In Section 2 we give simple forms for $L_v(a,b)$ and $I_v(a,b)$ and mean-inequalities are obtained in a fast and nice way. We also deduce two other weighted means from $L_v(a,b)$ and $I_v(a,b)$. Section 3 is devoted to investigate a general approach in service of weighted means. We then obtain more weighted means in another point of view. Section 4 displays the operator version of the previous weighted means as well as their related inequalities. In Section 5 we recall the standard power means known in the literature and we use, in Section 6, our approach for obtaining some new weighted means associated to the previous power means.

\section{\bf Another point of view for defining $L_v(a,b)$ and $I_v(a,b)$}

We preserve the same notations as in the previous section. The expressions \eqref{1} and specially \eqref{2} seem to be hard in computational context. We will see that we can rewrite them in other forms having convex characters.

\subsection{\bf Simple forms of $L_v(a,b)$ and $I_v(a,b)$.} The key idea of this section turns out the following result.

\begin{theorem}\label{thM}
For any $a,b>0$ and $v\in[0,1]$ we have
\begin{equation}\label{9}
L_v(a,b)=(1-v)L\big(a\sharp_vb,a\big)+vL\big(a\sharp_vb,b\big)=:L\big(a\sharp_vb,a\big)\nabla_vL\big(a\sharp_vb,b\big).
\end{equation}
\begin{equation}\label{10}
I_v(a,b)=\left(I\big(a\nabla_vb,a\big)\right)^{1-v}\left(I\big(a\nabla_vb,b\big)\right)^v
=:I\big(a\nabla_vb,a\big)\sharp_vI\big(a\nabla_vb,b\big).
\end{equation}
\end{theorem}
\begin{proof}
Starting from the middle expression of  \eqref{9} and using the definition of $L(a,b)$ and $a\sharp_vb$ we get the desired result after simple algebraic manipulations. By similar way we get \eqref{10}. The details are straightforward and therefore omitted here.
\end{proof}

In what follows we will see that the inequalities \eqref{3} and \eqref{4} can be immediately deduced from \eqref{9} and \eqref{10}, respectively. In fact we will prove more.

\begin{theorem}\label{thM2}
Let $a,b>0$ and $v\in[0,1]$. Then we have
\begin{equation}\label{11}
a\sharp_vb\leq\Big(a\sharp_{\frac{v}{2}}b\Big)\nabla_v\Big(a\sharp_{\frac{1+v}{2}}b\Big)\leq L_v(a,b)\leq(a\sharp_vb)\nabla(a\nabla_vb)\leq a\nabla_v b.
\end{equation}
\begin{equation}\label{12}
a\sharp_vb\leq\big(a\nabla_vb\big)\sharp\big(a\sharp_vb\big)\leq I_v(a,b)\leq\big(a\nabla_{\frac{v}{2}}b\big)\sharp_v\big(a\nabla_{\frac{1+v}{2}}b\big)\leq a\nabla_v b.
\end{equation}
\end{theorem}
\begin{proof}
The two right inequalities of \eqref{11} are those of \eqref{3}. We will prove them again by using \eqref{9}. Indeed, \eqref{9} with the help of \eqref{0} and then \eqref{02} yields
$$L_v(a,b)\leq(1-v)(a\sharp_vb)\nabla a+v(a\sharp_vb)\nabla b=(a\sharp_vb)\nabla(a\nabla_vb)\leq a\nabla_vb.$$
We now prove the two left inequalities of \eqref{11}. Again, \eqref{9} with \eqref{0} and then \eqref{02}, implies that
\begin{multline*}
L_v(a,b)\geq(1-v)(a\sharp_vb)\sharp a+v(a\sharp_vb)\sharp b=(1-v)a\sharp_{\frac{v}{2}}b+v\;a\sharp_{\frac{1+v}{2}}b\\
=(a\sharp_vb)^{1/2}\big(\sqrt{a}\nabla_v\sqrt{b}\big)\geq (a\sharp_vb)^{1/2}\big(\sqrt{a}\sharp_v\sqrt{b}\big)=a\sharp_vb.
\end{multline*}
We left to the reader the routine task for proving \eqref{12} in a similar manner.
\end{proof}

\begin{remark}
(i) From \eqref{9} and \eqref{10} we immediately see that $L_v(a,b)=L_{1-v}(b,a)$ and $I_v(a,b)=I_{1-v}(b,a)$, for any $a,b>0$ and $v\in[0,1]$.\\
(ii) From \eqref{9} and \eqref{10}, it is immediate to see that $L_v(a,b)$ and $I_v(a,b)$ are binary means in the sense that they satisfy the conditions itemized in \cite{PSMA}.\\
(iii) The inequalities \eqref{11} and \eqref{12} give alternative simple proofs for \cite[Corollary 2.2]{FM2020} and \cite[Corollary 2.3]{FM2020}, respectively..
\end{remark}

In order to emphasize even more the importance of \eqref{9} and \eqref{10} we will present below more results. These results investigate some inequalities refining the right inequalities in \eqref{11} and \eqref{12}. We need the following lemma.

\begin{lemma}
Let $a>0$ be fixed. Then the real-functions $x\longmapsto L(a,x)$ and $x\longmapsto I(x,a)$ are (strictly) concave for $x>0$.
\end{lemma}
\begin{proof}
It is a simple exercise of Real Analysis.
\end{proof}

\begin{theorem}\label{thM3}
For any $a,b>0$ and $v\in[0,1]$ we have
\begin{equation}\label{13}
L_v(a,b)\leq L\big(a\sharp_vb,a\nabla_vb\big)\leq I\big(a\sharp_vb,a\nabla_vb\big)\leq\big(a\sharp_vb\big)\nabla\big(a\nabla_vb\big).
\end{equation}
\begin{equation}\label{14}
a\sharp_vb\leq\big(a\nabla_v(a\sharp b)\big)\sharp_v\big((a\sharp b)\nabla_vb\big)\leq I_v(a,b).
\end{equation}
\end{theorem}
\begin{proof}
We prove the first inequality in \eqref{13}. Since the map $x\longmapsto L(a\sharp_vb,x)$ is concave for $x>0$ then \eqref{9} yields
$$L_v(a,b)\leq L\big(a\sharp_vb,(1-v)a+vb\big)=L\big(a\sharp_vb,a\nabla_vb\big).$$
The second and third inequalities of \eqref{13} follow from \eqref{0}.

To prove the second inequality of \eqref{14}, we write by using the previous lemma
$$I\big(a\nabla_vb,a\big)=I\big((1-v)a+vb,a\big)\geq(1-v)I(a,a)+vI(b,a)=(1-v)a+vI(a,b).$$
This, with the fact that $I(a,b)\geq a\sharp b$, implies that $I\big(a\nabla_vb,a\big)\geq a\nabla_v(a\sharp b)$. Similarly, we show that
$I\big(a\nabla_vb,b\big)\geq(a\sharp b)\nabla_vb$. This, with \eqref{10}, yields the second inequality of \eqref{14}. To prove the first inequality of \eqref{14} we write
$$\big(a\nabla_v(a\sharp b)\big)\sharp_v\big((a\sharp b)\nabla_vb\big)\geq\big(a\sharp_v(a\sharp b)\big)\sharp_v\big((a\sharp b)\sharp_vb\big)=a\sharp_vb,$$
after a simple computation. The proof is finished.
\end{proof}

\subsection{\bf Two other weighted means.} A natural question arises from the previous subsection: do we have a weighted mean $M_v(a,b)$
such that
\begin{equation}\label{FUR}
M_v(a,b) = M(a\sharp_vb,a)!_v M(a\sharp_ vb,b)?
\end{equation}

We can also put the following question: do we have a weighted mean $P_v(a,b)$ such that
\begin{equation}\label{FUR2}
P_v(a,b) = P(a!_vb,a)\sharp_v P(a!_vb,b)?
\end{equation}

In what follows we will answer the two preceding questions. Recall that $L^*$ denotes the dual of the logarithmic mean $L$ and $L_v^*$ is the dual of the weighted logarithmic mean $L_v$, as previously defined. Similar sentence for $I^*$ and $I_v^*$. We will establish the following result.

\begin{theorem}\label{thS}
For any $a,b>0$ and $v\in[0,1]$ we have
\begin{equation}\label{Lvd}
L_v^*(a,b)=L^*\big(a\sharp_vb,a\big)!_vL^*\big(a\sharp_vb,b\big).
\end{equation}
\begin{equation}\label{Ivd}
I_v^*(a,b)=I^*\big(a!_vb,a\big)\sharp_vI^*\big(a!_vb,b\big).
\end{equation}
\end{theorem}
\begin{proof}
We can of course assume that $v\in(0,1)$. If in \eqref{9} we replace $a$ and $b$ by $a^{-1}$ and
$b^{-1}$, respectively, then we get
$$L_v(a^{-1},b^{-1})=(1-v)L\big(a^{-1}\sharp_vb^{-1},a^{-1}\big)+vL\big(a^{-1}\sharp_vb^{-1},b^{-1}\big).$$
Taking the inverses side by side and using the definition of the weighted harmonic mean we infer that
\begin{equation}\label{Ld}
\big(L_v(a^{-1},b^{-1})\big)^{-1}=\Big(L\big(a^{-1}\sharp_vb^{-1},a^{-1}\big)\Big)^{-1}!_v\Big(L\big(a^{-1}\sharp_vb^{-1},b^{-1}\big)\Big)^{-1}.
\end{equation}
Now, let us set
\begin{equation}\label{Mv}
M(a\sharp_vb,a):=\Big(L\big(a^{-1}\sharp_vb^{-1},a^{-1}\big)\Big)^{-1}.
\end{equation}
If $v\in(0,1)$ is fixed, for any $a>0$ and $x>0$ it is easy to see that there exists a unique $b>0$ such that $a\sharp_vb=x$. This means that $M$ is well-defined by \eqref{Mv}. Further, if we remark that $a\sharp_vb=x$ implies $x^{-1}=a^{-1}\sharp_vb^{-1}$ then \eqref{Mv} becomes
$$M(x,a)=\big(L(x^{-1},a^{-1})\big)^{-1}:=L^*(x,a).$$
It follows that $M$ is the dual of the logarithmic mean $L$. Following \eqref{Ld} and \eqref{FUR}, the associated weighted mean $M_v$ of $M$ is such that
$$M_v(a,b)=\big(L_v(a^{-1},b^{-1}\big)^{-1}:=L_v^*(a,b),$$
i.e. $M_v(a,b)$ is the dual of the weighted logarithmic mean $L_v$. We left to the reader the task for proving \eqref{Ivd} in a similar manner.
\end{proof}

\begin{remark}
After this, let us observe the following question: is $L_v$ the unique weighted mean satisfying \eqref{9}? In the next section, we will answer this question via a general point of view. Similar question can be put for \eqref{10}, \eqref{Lvd} and \eqref{Ivd}.
\end{remark}

\section{\bf Weighted means in a general point of view}

As already pointed before, we will investigate here a study that shows how to construct some weighted means in a general point of view.

\subsection{\bf Position of the problem.} Let $a,b>0$ and $v\in[0,1]$. Let $p_v$ and $q_v$ be two weighted means. We write $ap_vb:=p_v(a,b)$ and $aq_vb:=q_v(a,b)$ for the sake of simplicity. As previous, $p:=p_{1/2}$ and $q:=q_{1/2}$ and we write $apb:=p(a,b)$ and $aqb=q(a,b)$. To fix the idea and for the first time, we can choose $p_v$ and $q_v$ among the three standard weighted means i.e. $ap_vb,aq_vb\in\big\{a\nabla_vb,a\sharp_vb,a!_vb\big\}$.

Our general problem reads as follows: do we have a weighted mean $M_v(a,b)$ such that
\begin{equation}\label{W1}
M_v(a,b)=M\big(ap_vb,a\big)q_vM\big(ap_vb,b\big)?
\end{equation}
To answer this question, it is in fact enough to justify that there exists one and only one (symmetric) mean $M$ such that
\begin{equation}\label{W2}
M(a,b)=M\big(apb,a\big)qM\big(apb,b\big).
\end{equation}
Indeed, $p_v(a,b)$ and $q_v(a,b)$ are given. Once the (symmetric) mean $M$ is found we obtain $M_v(a,b)$ by substituting $M(a,b)$ in \eqref{W1}.

Note that if $ap_vb,aq_vb\in\big\{a\nabla_vb,a\sharp_vb,a!_vb\big\}$ then we have nine cases. Theorem \ref{thM} answers the previous question for $(p_v,q_v)=(\sharp_v,\nabla_v)$ and $(p_v,q_v)=(\nabla_v,\sharp_v)$ while Theorem \ref{thS} answers the question for $(p_v,q_v)=(\sharp_v,!_v)$ and $(p_v,q_v)=(!_v,\sharp_v)$.

Our aim here is to answer the previous question in its general form. We need to recall some notions and results as background material that we will summarize in the next subsection.

\subsection{\bf Stable and stabilizable means.} We recall here in short the concept of stable and stabilizable means introduced in \cite{RAI0,RAI2,RAI4}. Let $m_1,m_2$ and $m_3$ be three given symmetric means. For all $a,b>0$, the resultant mean-map of $m_1,m_2$ and $m_3$ is defined by, \cite{RAI0}
$${\mathcal
R}(m_1,m_2,m_3)(a,b)=m_1\Big(m_2\big(a,m_3(a,b)\big),m_2\big(m_3(a,b),b\big)\Big).$$
A symmetric mean $m$ is called stable if ${\mathcal R}(m,m,m)=m$ and stabilizable if there exist two nontrivial stable means $m_1$
and $m_2$ such that ${\mathcal R}(m_1,m,m_2)=m$. We
then say that $m$ is $(m_1,m_2)$-stabilizable. If $m$ is stable then so is $m^*$, and
if $m$ is $(m_1,m_2)$-stabilizable then $m^*$ is
$(m_1^*,m_2^*)$-stabilizable. The tensor product of
$m_1$ and $m_2$ is the map, denoted $m_1\otimes m_2$,
defined by
$$\forall a,b,c,d>0\;\;\;\;\;\;\; m_1\otimes
m_2(a,b,c,d)=m_1\Big(m_2(a,b),m_2(c,d)\Big).$$
A symmetric mean $m$ is called cross mean if the map $m^{\otimes
2}:=m\otimes m$ is symmetric in its four variables. Every cross mean is stable, see \cite{RAI0}, and the converse still an open problem.

It is worth mentioning that, the operator version of the previous concepts as well as their related results has been investigated in a detailed manner in \cite{RAI5}. It has been proved there that every cross operator mean is stable but the converse does not in general hold provided that the Hilbert operator-space is of dimension greater than $2$.

The following results will be needed later, see \cite{RAI0,RAI2,RAI4}.

\begin{theorem}\label{thST}
(i) The arithmetic, geometric and harmonic means are cross means and so they are stable.\\
(ii) The logarithmic mean $L$ is $(!,\nabla)$-stabilizable and
$(\nabla,\sharp)$-stabilizable
while the identric mean $I$ is $(\sharp,\nabla)$-stabilizable.\\
(iii) The mean $L^*$ is $(\nabla,!)$-stabilizable and
$(!,\sharp)$-stabilizable while $I^*$ is $(\sharp,!)$-stabilizable.
\end{theorem}

For more examples and properties about stable and stabilizable means we can consult \cite{RAI0,RAI2,RAI4,RAI5}. See also Section 5 below.

\begin{theorem}\label{thCR}
Let $m_1$ and $m_2$ be two symmetric means such that $m_1\leq m_2$ (resp. $m_2\leq m_1$). Assume that $m_1$ is a strict cross mean. Then there exists one and only one $(m_1,m_2)$-stabilizable mean $m$ such that $m_1\leq m\leq m_2$ (resp. $m_2\leq m\leq m_1$).
\end{theorem}

\subsection{\bf The main result.} Now, we are in the position to answer our previous question as recited in the following result.

\begin{theorem}\label{thMS}
Let $a,b>0$ and $v\in[0,1]$. Let $p_v$ and $q_v$ be two weighted means such that $p:=p_{1/2}$ and $q:=q_{1/2}$ are stable. Assume that $q$ is a strict cross mean. Then there exists one and only one weighted mean $M_v(a,b)$ such that \eqref{W1} holds. Further, $M:=M_{1/2}$ is the unique $(q,p)$-stabilizable mean.
\end{theorem}
\begin{proof}
As already pointed before, it is enough to consider \eqref{W2}. Following the previous subsection, \eqref{W2} can be written as
$$M={\mathcal R}\big(q,M,p\big).$$
This means that $M$ is $(q,p)$-stabilizable. According to Theorem \ref{thCR}, such $M$ exists and is unique. Since $p_v$ and $q_v$ are given, we then deduce the existence and uniqueness of $M_v$ satisfying \eqref{W1}. The proof is finished.
\end{proof}

Following Theorem \ref{thST}, the symmetric means $a\nabla b, a\sharp b,a!b$ are cross means and so stable. From the preceding theorem we immediately deduce the following corollary.

\begin{corollary}
If $p_v,q_v\in\{\nabla_v,\sharp_v,!_v\}$ then we have the same conclusion as in the previous theorem.
\end{corollary}

The condition $p_v,q_v\in\{\nabla_v,\sharp_v,!_v\}$ includes exactly nine cases. The following examples discuss these cases in details.

\begin{example}
(i) Assume that $(p_v,q_v)=(\sharp_v,\nabla_v)$. Theorem \ref{thMS} implies that $M$ is the unique $(\nabla,\sharp)$-stabilizable mean and so Theorem \ref{thST} gives $M=L$. The related weighted mean is the weighted logarithmic mean $L_v$ given by \eqref{9}.\\
(ii) Assume that $(p_v,q_v)=(\nabla_v,\sharp_v)$. By similar way as previous, $M=I$ and $M_v=I_v$ is given by \eqref{10}.\\
(iii) Similarly, if $(p_v,q_v)=(\sharp_v,!_v)$ then $M=L^*$ and $M_v$ is given by \eqref{Lvd}. If $(p_v,q_v)=(!_v,\sharp_v)$ then $M=I^*$ and $M_v$ is given by \eqref{Ivd}.
\end{example}

\begin{example}
For the three cases $p_v=q_v\in\{\nabla_v,\sharp_v,!_v\}$ it is not hard to check that $M_v=p_v=q_v$ and so $M=p=q$. We can show this separately for every case by checking \eqref{W1} or use Theorem \ref{thMS} when combined with Theorem \ref{thST}. The details are immediate and therefore omitted here for the reader.
\end{example}

We have two cases left to see, namely $(p_v,q_v)=(\nabla_v,!_v)$ and $(p_v,q_v)=(!_v,\nabla_v)$, which we will discuss in the two following examples, respectively.

\begin{example}
Assume that $(p_v,q_v)=(\nabla_v,!_v)$. Following Theorem \ref{thMS}, $M$ is the unique $(!,\nabla)$-stabilizable mean and by Theorem \ref{thST} one has $M=L$. The related weighted mean $M_v$ is given by
\begin{equation}\label{aa}
M_v(a,b)=L\big(a\nabla_vb,a\big)!_vL\big(a\nabla_vb,b\big)
\end{equation}
By construction, we have $M=L$ the logarithmic mean. From \eqref{aa} we can check again that $M:=M_{1/2}=L$. In another word, the weighted mean $M_v(a,b)$ defined by \eqref{aa} is a second weighted logarithmic mean which we will denote by ${\mathcal L}_v$. Its explicit form is given by
$${\mathcal L}_v(a,b)=\frac{b-a}{\frac{1-2v}{v(1-v)}\log(a\nabla_vb)+\frac{v}{1-v}\log b-\frac{1-v}{v}\log a},\;\; {\mathcal L}_v(a,a)=a.$$
\end{example}

\begin{example}
Assume that $(p_v,q_v)=(!_v,\nabla_v)$. By similar way as previous, we show that the associated mean $M$ is here given by $M=L^*$ the dual logarithmic mean. The associated weighted mean $M_v$ is defined by
$$M_v(a,b)=L^*\big(a\nabla_vb,a\big)!_vL^*\big(a\nabla_vb,b\big).$$
Also, from this latter relation we can verify that $M:=M_{1/2}=L^*$ and $M_v={\mathcal L}_v^*$, with
$${\mathcal L}_v^*(a,b)=\frac{ab}{b-a}\left\{\frac{1-2v}{v(1-v)}\log(a!_vb)+\frac{v}{1-v}\log b-\frac{1-v}{v}\log a\right\},\;\; {\mathcal L}_v^*(a,a)=a.$$
The details are immediate and therefore omitted here.
\end{example}

The previous examples are summarized in TABLE \ref{table_01}.

\begin{table}[htbp]
\begin{center}
\caption{The weighted means $M_v$ constructed by $p_v$ and $q_v$.}\label{table_01}
\begin{tabular}{|c||c|c|c|}\hline
\backslashbox{$p_v$}{$q_v$} & \qquad $\nabla_v\,\,\,\,\,\,\,\,\,\,\,$ &  \qquad$\sharp_v\,\,\,\,\,\,\,\,\,\,\,$ & \qquad $!_v\,\,\,\,\,\,\,\,\,\,\,$	 \\
 \hline\hline
$\nabla_v$& $\nabla_v$ & $I_v$ & ${\mathcal L}_v$ \\
 \hline
$\sharp_v$  & $L_v$ & $\sharp_v$ & $L_v^*$ \\
  \hline
$!_v$  & ${\mathcal L}^*_v$ & $I^*_v$ & $!_v$ \\
  \hline
\end{tabular}
\end{center}
\end{table}




\section{\bf Operator Version}

The operator version of the previous weighted means as well as their related operator inequalities have been also discussed in \cite{PSMA}. By using their approach for operator monotone functions and referring to the Kubo-Ando theory \cite{KA}, they studied the analogs of $L_v(a,b)$ and $I_v(a,b)$ when the positive real numbers $a$ and $b$ are replaced by positive invertible operators.

Here, and with \eqref{9} and \eqref{10}, we don't need any more tools for giving in an explicit setting the operator versions of $L_v(a,b)$ and $I_v(a,b)$. Before exploring this, let us recall a few basic notions about operator means.

Let $H$ be a complex Hilbert space and let ${\mathcal B}(H)$ be the $\mathbb{C}^*$-algebra of bounded linear operators acting on $H$. The notation ${\mathcal B}^{+*}(H)$ refers to the open cone of all (self-adjoint) positive invertible operators in ${\mathcal B}(H)$. As usual, the notation $A\leq B$ means that $A,B\in{\mathcal B}(H)$ are self-adjoint and $B-A$ is positive semi-definite. A real-valued function $f$ on a nonempty $J$ of ${\mathbb R}$ is said to be operator monotone if and only if $A\leq B$ implies $f(A)\leq f(B)$ for self-adjoint operators $A$ and $B$ whose spectral $\sigma(A),\sigma(B)\subset J$. As usual, $f(A)$ is defined by the techniques of functional calculus. For further details about operator monotone functions we can consult \cite{BEP,IZN,UWYY,UDA} and the related references cited therein. Some examples of operator monotone functions will be considered below.

Following the Kubo-Ando theory \cite{KA}, there exists a unique one-to-one correspondence
between operator means and operator monotone functions. More precisely, an operator mean $m$ in the Kubo-Ando sense is such that
\begin{equation}\label{KA}
AmB=A^{1/2}f_m\big(A^{-1/2}BA^{-1/2}\big)A^{1/2},\; f_m(1)=1
\end{equation}
for some positive monotone increasing function $f_m$ on $(0,\infty)$. The function $f_m$ in \eqref{KA} is called the representing
function of the operator mean $m$. An operator mean in the Kubo-Ando sense is called operator monotone mean.

Let $A,B\in{\mathcal B}^{+*}(H)$ and $v\in[0,1]$. As standard examples of operator monotone means, the following
\begin{multline*}
A\nabla_{v}B=(1-v)A+vB,\;A\sharp_{v}B=A^{1/2}\Big(A^{-1/2}BA^{-1/2}\Big)^{v}A^{1/2},\\
A!_{v}B=\Big((1-v)A^{-1}+vB^{-1}\Big)^{-1}
\end{multline*}
are known  in the literature as the weighted arithmetic mean, the weighted geometric mean and the weighted harmonic mean of $A$ and $B$, respectively. If $v=1/2$ they are simply denoted by $A\nabla B,\; A\sharp B$ and $A!B$, respectively. The previous operator means satisfy the following double inequality
\begin{equation}\label{20}
A!_{v}B\leq A\sharp_{v}B\leq A\nabla_{v}B.
\end{equation}
The weighted logarithmic mean and the weighted identric mean of $A$ and $B$ can be, respectively, defined through:
\begin{equation}\label{22}
L_v(A,B)=A^{1/2}F_{L_v}\big(A^{-1/2}BA^{-1/2}\big)A^{1/2},\;\; F_{L_v}(x)=L_v(1,x),
\end{equation}
\begin{equation}\label{23}
I_v(A,B)=A^{1/2}F_{I_v}\big(A^{-1/2}BA^{-1/2}\big)A^{1/2},\;\; F_{I_v}(x)=I_v(1,x).
\end{equation}
For $v=1/2$, they are simply denoted by $L(A,B)$ and $I(A,B)$, respectively. From \eqref{0}, with the help of \eqref{KA}, we can immediately see that the chain of inequalities
\begin{equation}\label{235}
A!B\leq A\sharp B\leq L(A,B)\leq I(A,B)\leq A\nabla B
\end{equation}
is also valid for any $A,B\in{\mathcal B}^{+*}(H)$ and $v\in[0,1]$.

For the sake of information, the logarithmic mean $L(A,B)$ previously defined can be also alternatively given by one of the following integral forms:
$$L(A,B)=\int_0^1A\sharp_tB\;dt=\Big(\int_0^1\big(A\nabla_tB\big)^{-1}dt\Big)^{-1}.$$

It is worth mentioning that \eqref{22} and \eqref{23} define $L_v(A,B)$ and $I_v(A,B)$ just in the theoretical context. To give the explicit forms of $L_v(A,B)$ and $I_v(A,B)$, analogs to those of \eqref{1} and \eqref{2}, by using \eqref{22} and \eqref{23}, appears to be not obvious and no result reaching from this way. However, according to Theorem \ref{thM} with \eqref{KA} we immediately deduce the following.

\begin{theorem}
For any $A,B\in{\mathcal B}^{+*}(H)$ and $v\in[0,1]$ we have
\begin{equation}\label{25}
L_v(A,B)=L\big(A\sharp_vB,A\big)\nabla_vL\big(A\sharp_vB,B\big).
\end{equation}
\begin{equation}\label{26}
I_v(A,B)=I\big(A\nabla_vB,A\big)\sharp_vI\big(A\nabla_vB,B\big).
\end{equation}
\end{theorem}

Since all the involved operators in \eqref{25} and \eqref{26} are operator means in the sense of \eqref{KA} then by Theorem \ref{thM2} we immediately deduce the following result as well.

\begin{theorem}
Let $A,B\in{\mathcal B}^{+*}(H)$ and $v\in[0,1]$. Then we have
\begin{equation}\label{27}
A\sharp_vB\leq\Big(A\sharp_{\frac{v}{2}}B\Big)\nabla_v\Big(A\sharp_{\frac{1+v}{2}}B\Big)\leq L_v(A,B)\leq(A\sharp_vB)\nabla(A\nabla_vB)\leq A\nabla_v B.
\end{equation}
\begin{equation}\label{28}
A\sharp_vB\leq\big(A\nabla_vB\big)\sharp\big(A\sharp_vB\big)\leq I_v(A,B)\leq\big(A\nabla_{\frac{v}{2}}B\big)
\sharp_v\big(A\nabla_{\frac{1+v}{2}}B\big)\leq A\nabla_v B.
\end{equation}
\end{theorem}

\begin{remark}
The operator inequalities \eqref{27} and \eqref{28} recover Theorem 2.4 and Theorem 3.2 of \cite{PSMA}, respectively.
\end{remark}

By the same arguments as previous, the operator version of Theorem \ref{thM3} is immediately given in the following statement.

\begin{theorem}
For any $A,B\in{\mathcal B}^{+*}(H)$ and $v\in[0,1]$ we have
\begin{equation}\label{29}
L_v(A,B)\leq L\big(A\sharp_vB,A\nabla_vB\big)\leq I\big(A\sharp_vB,A\nabla_vB\big)\leq\big(A\sharp_vB\big)\nabla\big(A\nabla_vB\big).
\end{equation}
\begin{equation}\label{30}
A\sharp_vB\leq\big(A\nabla_v(A\sharp B)\big)\sharp_v\big((A\sharp B)\nabla_vB\big)\leq I_v(A,B).
\end{equation}
\end{theorem}

\section{\bf Power symmetric means}

This section deals with some weighted means for power symmetric means in one or two parameters. Let $a,b>0$ and $p,q$ be two real numbers. We recall the following:

$\bullet$ The power binomial mean defined by:
\begin{equation}\label{7}
B_p(a,b)=\Big(\frac{a^p+b^p}{2}\Big)^{1/p}
\end{equation}
This includes the particular cases $B_1(a,b)=a\nabla b,\;
B_0(a,b):=\lim\limits_{p\rightarrow0}B_p(a,b)=a\sharp b$ and $B_{-1}(a,b)=a!b$. Note that $B_{1/2}(a,b)=(a\nabla b)\nabla(a\sharp b)$.

$\bullet$ The power logarithmic mean defined by:
\begin{equation}\label{70}
L_p(a,b):=L_p=\left(\displaystyle\frac{a^{p+1}-b^{p+1}}{(p+1)(a-b)}\right)^{1/p},\; L_p(a,a)=a.
\end{equation}
We have $L_{-2}(a,b)=a\sharp b,\;L_{-1}(a,b)=L(a,b),\; L_0(a,b)=I(a,b),\; L_1(a,b)=a\nabla b$.

$\bullet$ The power difference mean given by:
\begin{equation}\label{71}
D_p(a,b):=D_p=\displaystyle{\frac{p}{p+1}\frac{a^{p+1}-b^{p+1}}{a^p-b^p}},\;
D_p(a,a)=a.
\end{equation}
In particular, $D_{-2}(a,b)=a!b,\; D_{-1}(a,b)=L^*(a,b),\; D_{-1/2}(a,b)=a\sharp b,\; D_0(a,b)=L(a,b)$ and  $D_1(a,b)=a\nabla b$.

$\bullet$ The power exponential mean defined as:
\begin{equation}\label{72}
I_p(a,b):=I_p=\exp\left(\displaystyle{-\frac{1}{p}+\frac{a^p\log\;a-b^p\log\;b}{a^p-b^p}}\right),\; I_p(a,a)=a.
\end{equation}
As special cases, $I_{-1}(a,b)=I^*(a,b),\; I_0(a,b)=a\sharp b$ and $I_1(a,b)=I(a,b)$.

$\bullet$ The second power logarithmic mean defined through:
\begin{equation}\label{73}
{\mathcal L}_p(a,b):={\mathcal L}_p=\left(\displaystyle{\frac{1}{p}\frac{b^p-a^p}{\log\;b-\log\;a}}\right)^{1/p},\;
l_p(a,a)=a.
\end{equation}
In particular, ${\mathcal L}_{-1}(a,b)=L^*(a,b),\; {\mathcal L}_0(a,b)=a\sharp b$ and ${\mathcal L}_1(a,b)=L(a,b)$.

$\bullet$ The previous power means are included in the so-called Stolarsky mean $S_{p,q}$:
\begin{equation}\label{74}
S_{p,q}:=S_{p,q}(a,b)=\displaystyle{\left(\frac{p}{q}\frac{b^q-a^q}{b^p-a^p}\right)^{1/(q-p)}},\;\; S_{p,q}(a,a)=a,
\end{equation}
in the sense that
\begin{equation}\label{75}
B_p=S_{p,2p},\; L_p=S_{1,p+1},\; D_p=S_{p,p+1},\; I_p=S_{p,p}:=\lim_{q\rightarrow p}S_{p,q},\; {\mathcal L}=S_{0,p}.
\end{equation}
All the previous power means are symmetric in $a$ and $b$. Also, remark that $S_{p,q}$ is symmetric in $p$ and $q$. Otherwise, the power binomial mean $B_p$ is stable for any $p\in{\mathbb R}$ and the following result holds, see \cite{RAI3}.

\begin{theorem}\label{thSM}
For any $p,q\in{\mathbb R}$, the Stolarsky mean $S_{p,q}$ is $\big(B_{q-p},B_p\big)$-stabilizable.
\end{theorem}

The previous theorem when combined with \eqref{75} and a simple argument of continuity immediately implies the following, see also \cite{RAI0}.

\begin{corollary}\label{thCSM}
For all real number $p$, the following assertions hold:\\
(i) The power mean $L_p$ is $(B_p,\nabla)$-stabilizable while $D_p$ is $(\nabla,B_p)$-stabilizable.\\
(ii) The power mean $I_p$ is $(\sharp,B_p)$-stabilizable while ${\mathcal L}_p$ is
$(B_{p},\sharp)$-stabilizable.
\end{corollary}

Now, let us observe the following remark which is of interest.

\begin{remark}\label{remS}
Since $S_{p,q}=S_{q,p}$ we can also say that $S_{p,q}$ is $\big(B_{p-q},B_q\big)$-stabilizable. This, with \eqref{75}, implies also that, $L_p$ is $(B_{-p},B_{p+1})$-stabilizable, $D_p$ is $(!,B_{p+1})$-stabilizable, ${\mathcal L}_p$ is $(B_{-p},B_p)$-staabilizable and no news for $I_p$. Obviously, (i) and (ii) of Corollary \ref{thCSM} are simpler than these latter statements.
\end{remark}

\section{\bf Some new weighted power means}

In this section we will investigate the weighted means of the previous power means. The weighted power binomial mean can be immediately given by
$$B_{p;v}(a,b)=\Big((1-v)a^p+vb^p\Big)^{1/p},\;\; B_{0,v}(a,b)=a\sharp_vb.$$
This, with the results presented in the preceding section, will allow us to construct some new weighted power means. Recall that, $m_v$ is called weighted mean if it satisfies the conditions: $m_v$ is a mean for any $v\in[0,1]$, $m:=m_{1/2}$ is a symmetric mean and $m_v(a,b)=m_{1-v}(b,a)$ for any $a,b>0$ and $v\in[0,1]$. We then say that $m_v$ is a $m$-weighted mean and $m$ is the symmetric mean of $m_v$. It is obvious that for any weighted mean $m_v$, its associated symmetric mean $m:=m_{1/2}$ is unique. However, for a given symmetric mean $m$ we can have two $m$-weighted means $m_v$ and $l_v$ i.e. $m_{1/2}=l_{1/2}=m$. For more explanation about this latter situation, see the examples below.

In a general context, we have the following result.

\begin{theorem}\label{thWm}
Let $m$ and $l$ be two stable means and let $M$ be $(l,m)$-stabilizable. Let $m_v$ and $l_v$ be the $m$-weighted mean and the $l$-weighted mean, respectively. Then the following
\begin{equation}\label{Wm}
M_v(a,b)={\mathcal R}\big(l_v,M,m_v\big)(a,b)=M\big(am_vb,a\big)l_vM\big(am_vb,b\big)
\end{equation}
is a $M$-weighted mean.
\end{theorem}
\begin{proof}
It is straightforward. The details are simple and therefore omitted here for the reader.
\end{proof}

Applying the previous simple result to the preceding power means, we will immediately obtain their associated weighted power means. We present these in the following examples. We begin by the $S_{p,q}$-weighted mean and we then deduce the other weighted power means as particular cases.

\begin{example}
By Theorem \ref{thSM}, $S_{p,q}$ is $(B_{q-p},B_p)$-stabilizable. By Theorem \ref{thWm}, an $S_{p,q}$-weighted mean is given by
\begin{equation}\label{zz}
S_{p,q;v}(a,b)=B_{q-p;v}\Big(S_{p,q}\big(B_{p;v}(a,b),a\big),S_{p,q}\big(B_{p;v}(a,b),b\big)\Big).
\end{equation}
Utilizing \eqref{70} with \eqref{7}, the explicit form of $S_{p,q;v}(a,b)$ is given by (for $a\neq b$)
$$S_{p,q;v}(a,b)=\left(\frac{p}{q}\frac{1}{b^p-a^p}\Big(\frac{1-v}{v}\big(B_{p;v}^q-a^q\big)
+\frac{v}{1-v}\big(b^q-B_{p;v}^q\big)\Big)\right)^{\frac{1}{q-p}},$$
provided that $p\neq 0,q\neq 0,p\neq q$, where we write $B_{p;v}:=B_{p;v}(a,b)$ for simplifying the writing.
The three cases $p=0,q=0$ and $p=q$ will be presented later.
\end{example}

\begin{example}
Since $S_{p,q}$ is also $(B_{p-q},B_q)$-stabilizable, see Remark \ref{remS}, another $S_{p,q}$-weighted mean is given by
$$S_{p,q;v}(a,b)=B_{p-q;v}\Big(S_{p,q}\big(B_{q;v}(a,b),a\big),S_{p,q}\big(B_{q;v}(a,b),b\big)\Big),$$
or, in explicit form, if $p\neq 0,q\neq 0,p\neq q$ and $a\neq b$,
$$S_{p,q;v}(a,b)=\left(\frac{q}{p}\frac{1}{b^q-a^q}\Big(\frac{1-v}{v}\big(B_{q;v}^p-a^p\big)
+\frac{v}{1-v}\big(b^p-B_{q;v}^p\big)\Big)\right)^{\frac{1}{p-q}}.$$
\end{example}

\begin{example}
(i) By Corollary \ref{thCSM}, $L_p$ is $(B_p,\nabla)$-stabilizable. By Theorem \ref{thWm}, the $L_p$-weighted mean is given by
$$L_{p;v}(a,b)=B_{p;v}\Big(L_p\big(a\nabla_vb,a\big),L_p\big(a\nabla_vb,b\big)\Big).$$
By \eqref{70} and \eqref{7}, or just using the relation $L_p=S_{1,p+1}$ with \eqref{zz}, we obtain the explicit form of $L_{p;v}$:
$$L_{p;v}(a,b)=\left(\frac{1}{(p+1)(b-a)}\Big(\frac{1-v}{v}\big((a\nabla_vb)^{p+1}-a^{p+1}\big)+
\frac{v}{1-v}\big(b^{p+1}-(a\nabla_vb)^{p+1}\big)\Big)\right)^{1/p}.$$
(ii) Similarly, since $D_p$ is $(\nabla,B_p)$-stabilizable, we then deduce that the $D_p$-weighted mean is given by
$$D_{p;v}(a,b)=D_p\big(B_{p;v}(a,b),a\big)\nabla_vD_p\big(B_{p;v}(a,b),b\big),$$
or in explicit form, with $B_{p;v}:=B_{p;v}(a,b)$,
$$D_{p;v}(a,b)=\frac{p}{p+1}\frac{1}{b^p-a^p}\left(\frac{1-v}{v}\Big(B_{p;v}^{p+1}-a^{p+1}\Big)+\frac{v}{1-v}\Big(b^{p+1}-B_{p;v}^{p+1}\Big)\right).$$
\end{example}

\begin{example}
By similar arguments as in the previous examples, we obtain:\\
(i) The ${\mathcal L}_p$-weighted mean is defined by
$${\mathcal L}_{p;v}(a,b)=B_{p;v}\Big({\mathcal L}_p\big(a\nabla_vb,a\big),{\mathcal L}_p\big(a\nabla_vb,b\big)\Big),$$
or in explicit form
$${\mathcal L}_{p;v}(a,b)=\left(\frac{p}{b^p-a^p}\Big(\frac{1-v}{v}\big(\log B_{p;v}-\log a\big)+
\frac{v}{1-v}\big(\log b-\log B_{p;v}\big)\Big)\right)^{-\frac{1}{p}}.$$
(ii) The $I_p$-weighted mean is given by
$$I_{p;v}(a,b)=I_p\Big(B_{p;v}(a,b),a\Big)\sharp_vI_p\Big(B_{p;v}(a,b),b\Big).$$
We left to the reader the task for giving the explicit form of this latter weighted power mean.
\end{example}

\section*{Acknowledgement}
The author (S.F.) was partially supported by JSPS KAKENHI Grant Number 16K05257.

\end{document}